\documentclass[times]{elsarticle}

\usepackage{a4,amssymb,amsthm,amscd,amsmath,amsfonts,units,comment,cleveref}
\usepackage{color}

\usepackage{enumerate}

\usepackage[normalem]{ulem}

\DeclareMathOperator{\End}{\mathsf{End}}
\newcommand{\N}{\mathbb{N}}

\newcommand{\id}{\mathsf{id}}
\newcommand{\ad}{\mathsf{ad}}
\newcommand{\la}{\langle}
\newcommand{\ra}{\rangle}

\newcommand{\nat}{\mathbb{N}}
\newcommand{\zz}{\mathbb Z}

\newcommand{\hh}{\mathbb H}
\newcommand{\dset}[2]{\D_{#1} #2}

\newcommand{\mg}[1]{#1^{\times}}
\newcommand{\sq}[1]{#1^{\times 2}}
\newcommand{\scg}[1]{\mg{#1}/\sq{#1}}
\newcommand{\sums}[1]{\sum\!{#1}^2}
\newcommand{\lla}{\la\!\la}
\newcommand{\rra}{\ra\!\ra}

\newcommand{\unl}{\underline}

\newcommand{\Sym}[1]{\mathsf{Sym}({#1})}
\newcommand{\Skew}[1]{\mathsf{Skew}({#1})}

\newcommand{\sign}{\mathsf{sign}}

\newcommand{\s}{\sigma}
\newcommand{\x}{\times}
\newcommand{\ox}{\otimes}

\newcommand{\type}{\mathsf{type}}
\newcommand{\ve}{\varepsilon}
\newcommand{\vf}{\varphi}
\newcommand{\vt}{\vartheta}

\renewcommand{\deg}{\mathsf{deg}}

\newcommand{\too}{\rightarrow}
\newcommand{\mapstoo}{\mapsto}

\newcommand{\Ad}{\mathsf{Ad}}

\newcommand{\an}{\mathsf{an}}

\newcommand{\qi}{\mid}
\newcommand{\qil}{\,\cdot\!\!\mid}
\newcommand{\qir}{\mid\!\!\cdot\,}
\newcommand{\qilr}{\,\cdot\!\!\mid\!\!\cdot\,}

\newcommand{\Sim}{\mathsf{Sim}}
\newcommand{\Int}{\mathsf{Int}}
\newcommand{\Trd}{\mathsf{Trd}}

\renewcommand{\deg}{\mathsf{deg}}
\renewcommand{\dim}{\mathsf{dim}}

\newcommand{\can}{\mathsf{can}}

\newcommand{\lra}{\rightarrow}
\newcommand{\half}{\mbox{$\frac{1}{2}$}}
\newcommand{\itr}{\mathsf{Tr}}
\newcommand{\G}{\mathsf{G}}
\newcommand{\D}{\mathsf{D}}
\newcommand{\M}{\mathsf{M}}
\newcommand{\tp}{\mathsf{t}}

\newcommand{\Z}{\mathsf{Z}}

\renewcommand{\leq}{\leqslant}
\renewcommand{\geq}{\geqslant}

\numberwithin{equation}{section}
\newtheorem{thm}[equation]{Theorem}
\newtheorem{prop}[equation]{Proposition}
\newtheorem{cor}[equation]{Corollary}

\newtheorem{lem}[equation]{Lemma}
\newtheorem{qu}[equation]{Question}

\theoremstyle{definition}

\newtheorem{ex}[equation]{Example}

\newtheorem{rem}[equation]{Remark}

\makeatletter
\def\@oddhead{K.J. Becher, T. Unger: Weakly hyperbolic involutions\hfill \thepage}
\def\@evenhead{Weakly hyperbolic involutions}
\def\ps@pprintTitle{%
     \let\@oddhead\@empty
     \let\@evenhead\@empty
     \def\@oddfoot{August 8, 2017\hfill}%
     \let\@evenfoot\@oddfoot}
\makeatother

\begin{document}

\begin{frontmatter}

\title{Weakly hyperbolic involutions}
\author[kjb]{Karim Johannes Becher}
\ead{karimjohannes.becher@uantwerpen.be}

\author[tu]{Thomas Unger}
\ead{thomas.unger@ucd.ie}

\address[kjb]{Universiteit Antwerpen, Departement Wiskunde--Informatica, Middelheim\-laan~1, 2020 Antwerpen, Belgium}

\address[tu]{School of Mathematics and Statistics, University College Dublin, Belfield,\linebreak Dublin~4, Ireland}

\begin{abstract}
\noindent
Pfister's Local-Global Principle states that a quadratic form over a (formally) real field is weakly hyperbolic (i.e.~represents a torsion element in the Witt ring) if and only if its total signature is zero.
This result extends naturally to the setting of central simple algebras with involution.
The present article provides a new proof of this result and extends it to the case of signatures at preorderings.
Furthermore the quantitative relation between nilpotence and torsion is explored for quadratic forms as well as for central simple algebras with involution.
\end{abstract}


\end{frontmatter}

\section{Introduction} 

Pfister's Local-Global Principle says that a regular quadratic form over a (formally) real field  represents a torsion element in the Witt ring if and only if its signature at each ordering of the field is zero. 
This result has been extended in \cite{LU} to  central simple algebras with involution. 

The theory of central simple algebras with involution is a natural extension of quadratic form theory. 
On the one hand many concepts and results from quadratic form theory
have been extended  to algebras with involution. 
On the other hand quadratic forms are used as tools in the study of algebras with involution. Examples include involution trace forms and spaces of similitudes. 

In this article we are interested in
weakly hyperbolic algebras with involution, a natural generalization  of torsion quadratic forms, which was considered first in \cite[Chap.~5]{Unger}.
In \cite{LU} such algebras with involution were characterized as those having trivial signature at all orderings of the base field, thus generalizing Pfister's Local-Global Principle.

We give a new exposition of this result, highlighting several new aspects, and obtain some extensions.
In particular, we provide bounds for the torsion order of nilpotent quadratic forms and extend this result to involutions.
We attempt to minimize the use of hermitian forms and treat algebras with involution as direct analogues of quadratic forms. 
We only consider fields of characteristic different from $2$ since the notion of weak hyperbolicity is only interesting in this case.

The structure of this article is as follows.
In Section~\ref{Sec:PLGP} we recall the necessary background material from the theory of quadratic forms 
and ordered fields as well as   Pfister's Local-Global Principle for quadratic forms in a generalized version, 
relative to  preorderings. 
We also touch on the quantitative aspect of the relation between nilpotence and torsion, using 
Lewis' annihilating polynomials. 

In Section~\ref{sec3} 
 we recall the basic terminology for algebras with involution,  consider their  relations to quaternion algebras and quadratic forms and study involution trace forms.

In Section~\ref{sec6} we treat the notion of hyperbolicity for algebras with involution and 
cite the relevant results about its behaviour under field extensions.

In Section~\ref{sec7} we turn to the study of algebras with involution over real fields.
In \Cref{P:rcai} we obtain a classification  over real closed fields.
We then provide  a uniform definition of signatures for involutions of both kinds  with respect to an ordering.
Signatures  were introduced in \cite{LT} for involutions of the first kind 
and in \cite{Queg} for involutions of the second kind, and both cases are treated in \cite[(11.10), (11.25)]{BOI}.

In Section~\ref{sec8} we give a new proof of the main result of \cite{LU}, 
an analogue of Pfister's Local-Global Principle for algebras with involution (\Cref{T:PLULG}). 
In \Cref{T:PLULG-Pre} we extend this result to a local-global principle for $T$-hyperbolicity with respect to a preordering~$T$.

In its original version for quadratic forms as well as in the generalized version for algebras with involution  Pfister's Local-Global Principle relates the hyperbolicity of tensor powers  to the hyperbolicity of multiples. For quadratic forms  this corresponds to the relation between nilpotence and torsion for an element of the Witt ring. In Section~\ref{sec9} 
we consider the quantitative aspect of this relation in the setting of algebras with involution. 

Some of the essential ideas contained in  Sections~\ref{sec7} and~\ref{sec8} germinated in the {MSc} thesis of Beatrix Bernauer \cite{BB}, prepared under the guidance of the first named author.

\section{Pfister's Local-Global Principle} 
\label{Sec:PLGP}

We refer to \cite{Lam}  and \cite{Scharlau} for the foundations of quadratic form theory over fields and the relevant terminology.
Let $K$ be a field of characteristic different from $2$.
We denote by $\mg{K}$ the multiplicative group of $K$, by $\sq{K}$ the subgroup of nonzero squares, and by $\sums{K}$ the subgroup of nonzero sums of squares in $K$.
If $\sums{K}=\sq{K}$ then $K$ is said to be \emph{pythagorean}.

By a \emph{quadratic form over $K$} (or just a \emph{form}) we mean a pair $(V,B)$ consisting of a finite-dimensional $K$-vector space $V$ and a regular symmetric $K$-bilinear form $B$. 
We consider quadratic forms up to isometry and use the equality sign to indicate that two forms are 
isometric. 
For a form $\vf=(V,B)$ over $K$ we write $\dset{K}{(\vf)}=\{a\in\mg{K} \mid  a=B(x,x)\}$.
We further abbreviate  $\D_K(m) =\D_K(m\x \la 1\ra)$, which is
the set of nonzero sums of $m$ squares in $K$.
Given $n\in\N$ and $a_1,\dots,a_n\in\mg{K}$, we use the standard notations $\la a_1,\dots,a_n\ra$ and 
$\lla a_1,\dots,a_n\rra=\la 1,-a_1\ra\ox\cdots\ox\la 1,-a_n\ra$ to denote diagonalized quadratic forms of dimension $n$ and $n$-fold Pfister forms, respectively.
Given a quadratic form $\vf$ over $K$ and  $m\in \nat$, we write $m\x \vf$ for the $m$-fold orthogonal sum $\vf \perp \cdots \perp \vf$ and further $\vf^{\ox m}$ for the $m$-fold tensor product $\vf\ox\cdots\ox\vf$.

Let $W\!K$ denote the Witt ring of $K$ and $I\!K$ its fundamental ideal. 
We sometimes write $[\vf]$ to denote the class in $W\!K$ given by a  form $\vf$ over $K$.

An \emph{ordering of $K$} is a set $P\subseteq K$ that is additively and multiplicatively closed and that satisfies $P\cup -P=K$ and $P\cap -P=\{0\}$.
Any such set $P$ is the  \emph{positive cone} $\{ x\in K\mid x\geq 0\}$ for a unique total order relation $\leq$ on $K$ that is compatible with the field operations.
Let $X_K$ denote the set of orderings of $K$. If the field $K$ has an ordering then we say that it is \emph{real}, otherwise \emph{nonreal}. By the Artin-Schreier Theorem the field $K$ is real if and only if  $-1\notin\sums{K}$. 

Let $T\subseteq K$ be additively and multiplicatively closed  with $\sq{K}\cup\{0\}\subseteq T$.
Then the set $T+xT=\{s+xt\mid s,t\in T\}$ is  additively and multiplicatively closed for any $x\in K$.
Moreover $\mg T=T\setminus \{0\}$ is a subgroup of $\mg{K}$ containing $\sums{K}$.
If further $-1\notin T$, then $T$ is called a \emph{preordering of $K$}. 
Any ordering is a preordering. 
Furthermore, if $T$ is a preordering of $K$, then so is $T+xT$ for any $x\in K\setminus-T$.
It follows from this that a preordering of $K$ is maximal with respect to inclusion if and only if it is an ordering. Hence, every preordering is contained in an ordering.
For a preordering $T$ of $K$ we set $$X_T=\{P\in\ X_K\mid T\subseteq P\}\,.$$

Given a quadratic form $\vf$ over $K$, we denote the signature of $\vf$ at an ordering $P\in X_K$ by $\sign_P(\vf)$ and obtain a map 
$\widehat{\vf} : X_K\too \zz, P\longmapsto \sign_P(\vf)$.
This gives rise to a ring homomorphism $$\sign : W\!K \too \zz^{X_K}, \vf\mapstoo \widehat{\vf}$$
called the \emph{total signature}.
If $K$ is nonreal, then $X_K=\emptyset$ and $\zz^{X_K}$ is the ring with one element.
Let $T$ be a fixed preordering of $K$. 
We write $$\sign_T: W\!K\too \zz^{X_T},\vf\mapstoo \widehat{\vf}|_{X_T}$$
and we denote the kernel of this homomorphism by $I_TK$.

Let $\vf$ be a quadratic form over $K$.
We say that $\vf$ is \emph{$T$-positive} if $\vf$ is nontrivial and  $\dset{K}{(\vf)}\subseteq \mg{T}$.
If $a_1,\dots,a_n\in\mg{K}$ are such that $\vf=\la a_1,\dots,a_n\ra$, then $\vf$ is $T$-positive if and only if $a_1,\dots,a_n\in\mg{T}$.
Hence, orthogonal sums and tensor products of $T$-positive forms are again $T$-positive.
We say that $\vf$ is \emph{$T$-isotropic} or \emph{$T$-hyperbolic} if there exists a $T$-positive form $\vt$ over $K$ such that $\vt\otimes\vf$ is isotropic or hyperbolic, respectively.

The following statement is a generalization of Pfister's Local-Global Principle, relative to a preordering (cf.   \cite[(1.26)]{LamOVQ}).

\begin{thm}[Pfister]
\label{T:pre-Pfister}
Let $T$ be a preordering of $K$.
The ideal $I_TK$ is generated by the classes of binary forms $\la 1,-t\ra$ with $t\in\mg{T}$.
Moreover, for a quadratic form $\varphi$ over $K$ the following statements are equivalent:
\begin{enumerate}[$(i)$]
\item $\sign_T(\varphi)=0$.
\item The form $\varphi$ is $T$-hyperbolic.
\item There exists a $T$-positive Pfister form $\tau$ over $K$ such that $\tau\otimes\varphi$ is hyperbolic.
\item There exist $r\geq 0$, $a_1,\dots,a_r\in\mg{K}$ and $t_1,\dots,t_r\in\mg{T}$ such that $\varphi$ is Witt equivalent to $\la a_1,-a_1t_1\ra\perp\dots\perp\la a_r,-a_rt_r\ra$.
\end{enumerate}
\end{thm}

A quadratic form $\vf$ over $K$ is said to be \emph{torsion}  or \emph{weakly hyperbolic} if $m\times \varphi$ is hyperbolic for some positive integer $m$.
The following is \cite[Satz 10 and Satz 22]{Pfi66}.

\begin{cor}[Pfister]\label{C:PLGP}
Assume that $K$ is real. For a quadratic form $\varphi$ over $K$ the following statements are equivalent:
\begin{enumerate}[$(i)$]
\item $\sign(\varphi)=0$.
\item The quadratic form $\varphi$ is weakly hyperbolic.
\item There exists $m\in\N$ such that $2^m\x \vf$ is hyperbolic.
\item There exists $n\in\N$ such that $\vf^{\ox n}$ is hyperbolic.
\item There exist $r\geq 0$, $a_1,\dots,a_r\in\mg{K}$ and $s_1,\dots,s_r\in \sums{K}$ such that $\varphi$ is Witt equivalent to $\la a_1,-a_1s_1\ra\perp\dots\perp\la a_r,-a_rs_r\ra$.
\end{enumerate}
\end{cor}

\begin{cor}[Pfister]\label{C:Scharlau-2prim}
The order of any torsion element in $W\!K$ is a $2$-power. 
\end{cor}

For $n\in\N$ let 
$$L_n(X) \,\,=\,\, \prod_{i=0}^n (X-n+2i) \in \mathbb{Z}[X].$$
Note that $L_n(-X)=(-1)^{n+1}\cdot L_n(X)$, so that the polynomial $L_n(X)$ is either even or odd.
In \cite{Lew}, Lewis showed that  these polynomials have a crucial property relating 
to  quadratic forms.
We include the short proof due to K.H.~Leung, which is mentioned in~\cite{Lew}.

\begin{thm}[Lewis]\label{T:Lewis}
Let $n\in \N$ and let $\vf$ be a  quadratic form of dimension $n$ over $K$. Then $L_n([\vf]) = 0$ in $W\!K$.
\end{thm}

\begin{proof}
As $(a\vf)^{\ox 2} = \vf^{\ox 2}$ for all $a\in\mg{K}$ and  $L_n(X)$ is either even or odd, we may scale $\vf$ and assume that $\vf = \vf'\perp\la 1\ra$  where $\vf'$ is a form of dimension $n-1$. 
Using the induction hypothesis for $\vf'$ we obtain that $L_{n-1}([\vf]-1)=L_{n-1}([\vf'])=0$.
In view of the equality $L_n(X) = (X+n)\cdot L_{n-1}(X-1)$ we conclude that $L_n([\vf])=0$.
\end{proof}

As shown in \cite{Lew} and \cite{Lew89} the last statement can be used to prove some statements about the Witt ring of quadratic forms which were obtained earlier by different methods, usually involving field extensions.
This holds in particular for \Cref{T:pre-Pfister} and \Cref{C:PLGP}. 
Another such example is the following statement. 

\begin{cor}[Scharlau]\label{C:zero-div}
Any zero-divisor of $W\!K$ lies in $I\!K$.
\end{cor}

\begin{proof}
Let $\alpha\in W\!K\setminus I\!K$.
By \Cref{T:Lewis} there exists $n\in\nat$ such that $\alpha$ is a zero of
$L_{2n+1}(X)=\prod_{i=0}^{n} (X^2-(2i+1)^2)$.
Hence, for $\beta\in W\!K$ with $\alpha\beta=0$ we have $m\beta=0$ for the odd integer $m=\prod_{i=0}^{n} (2i+1)^2$, thus $\beta=0$ by \Cref{C:Scharlau-2prim}.
\end{proof}

\begin{cor}
\label{C:alpha-mult}
Let $n\in\nat$ and let $\vf$ be a quadratic form of dimension $2n$ over~$K$.
Then $2^{2n-1}n!(n-1)!\cdot[\vf]$ is a multiple of $[\vf]^2$ in $W\!K$.
\end{cor}

\begin{proof}
We may scale $\vf$ and assume that $\vf=\la 1\ra\perp\vf'$ where $\vf'$ is a form of dimension $2n-1$.
Then $L_{2n-1}([\vf'])=0$ by \Cref{T:Lewis}. It follows that $[\vf]$ is a zero of the polynomial
\[\smash{L_{2n-1}(X-1)=(X-2n)X\prod_{i=1}^{n-1} (X^2-4i^2).}\]
This  implies the statement.
\end{proof}

For $n\in \N$ we denote by $d(n)$ the number of occurrences of the digit~$1$ in the binary representation of $n$. Note that $d(2n)=d(n)$ and $d(2n+1)=d(n)+1$.
In  \cite[\S4.4]{Conc} the following observation is attributed to Legendre.

\begin{prop}\label{P:Legendre}
For $n\in\nat$ the largest $2$-power dividing $n!$ is equal to $2^{n-d(n)}$.
\end{prop}

\begin{proof}  Let $n\in\nat$.
The largest $2$-power dividing $n$ is  $2^m$ where $m$ is the number of consecutive digits~$1$ at the end of the binary representation of $n-1$, whereby  $m=d(n-1)-d(n)+1$. 
Hence the largest $2$-power dividing $n!$ is $2^k$, where $k= \sum_{i=1}^n \bigl(d(i-1)-d(i)+1\bigr)=n-d(n)$.
\end{proof}

For $n\geq 1$ we  set
$\Delta(n)=2n-1-d(n)-d(n-1)$. 

\begin{thm}
\label{P:Lou}
Let $\vf$ and $\pi$ be quadratic forms over $K$ such that $\vf\ox \vf\ox \pi$ is hyperbolic.
Then $2^{\Delta(n)}\x \vf\ox\pi$ is hyperbolic for $n=\dim(\vf)$.
\end{thm}

\begin{proof}
Assume that $\pi$ is not hyperbolic, as otherwise the statement is trivial. 
Then $n=\dim(\vf)=2k$ for some $k\in\nat$, by \Cref{C:zero-div}.
It follows from \Cref{C:alpha-mult} that the form $2^{2k-1}k!(k-1)!\times \varphi\otimes\pi$ is hyperbolic.
By \Cref{C:Scharlau-2prim} the order of $[\vf\otimes\pi]$ in $WK$ is a $2$-power and in view of \Cref{P:Legendre} the largest $2$-power dividing $2^{2k-1}k!(k-1)!$
is $2^m$ where $$m=(2k-1)+k-d(k)+k-1-d(k-1)=4k-1-d(2k)-d(2k-1)=\Delta(n).$$ 
Hence $2^{\Delta(n)}\x \vf\ox\pi$ is hyperbolic.
\end{proof}

\begin{rem}
In view of \Cref{P:Lou} we may define a function $g: \N\too \N$ in the following way: for $k\in \N$, let $g(k)$ be the smallest number $m\in \N$ such that, for any quadratic form $\vf$ of dimension $2k$ over an arbitrary field of characteristic different from $2$ for which  $\vf\ox\vf$ is hyperbolic, also $2^m\x \vf$ is hyperbolic.
Applying \Cref{P:Lou} with $\pi=\la 1\ra$ yields that  $g(k)\leq \Delta(2k)=4k-2-d(k)-d(k-1)$. This bound, however, does not seem to be optimal for $k>1$. In fact, it is not difficult to show that $g(2)=g(3)=2$.
\end{rem}

\section{Algebras with involution} 
\label{sec3}

Our general references for the theory of central simple algebras and their involutions are \cite{BOI} and \cite[Chap.~8]{Scharlau}.
We fix some terminology.

Let $K$ be a field
and let $A$ be a $K$-algebra.
We call $A$ a \emph{$K$-division algebra} if every non-zero element in $A$ is invertible.
We denote by $\Z(A)$ the centre of $A$.
We call the $K$-algebra $A$ \emph{central simple} if it is finite-dimensional over $K$, simple as a ring and such that $\Z(A)=K$.
A \emph{$K$-involution on $A$} is a $K$-linear map $\sigma : A\too A$ such that $\s(xy)=\s(y)\s(x)$ for all $x,y\in A$ and $\sigma\circ\sigma=\id_A$.

Assume now that $A$ is a central simple $K$-algebra. 
Wedderburn's Theorem says that in this case $\dim(A)=n^2$ for a positive integer $n$, called the \emph{degree of $A$} and denoted $\deg(A)$, and further that $A$ is isomorphic to a matrix algebra over a central simple $K$-division algebra $D$ unique up to $K$-isomorphism;
one says that $A$ is \emph{split} if $D=K$. If $A$ and $A'$ are matrix algebras over the same central simple $K$-division algebra, then $A$ and $A'$ are called \emph{Brauer equivalent} and we write $A\sim A'$.

A \emph{$K$-algebra with involution} is a pair  $(A,\s)$
where $A$ is a finite-dimensional $K$-algebra and  $\s$ is a $K$-involution on $A$ such that  $K=\{x\in \Z(A)\mid \s(x)=x\}$ and such that either $A$ is simple or $A$ is a product of two simple $K$-algebras that are mapped to each other by $\s$.
We will often denote a $K$-algebra with involution by a single capital Greek letter.

Let $(A,\s)$ be a $K$-algebra with involution.
We have either $\Z(A)=K$ or $\Z(A)$ is a quadratic \'etale extension of $K$.
We say that  $(A,\s)$ is of the \emph{first} or \emph{second kind}, depending on whether $[\Z(A):K]$ is $1$ or $2$, respectively.
Note that $A$ is simple if and only if $\Z(A)$ is a field.
If $\Z(A)$ is not a field then it is isomorphic to $F\times F$; 
this can only occur if $(A,\s)$ is of the second kind.
We have $\dim_K(A)=[\Z(A):K]\cdot n^2$ for a positive integer $n\in\nat$, which we call the \emph{degree of $(A,\s)$}; if $A$ is simple, then $\deg(A,\s)$ is just the degree of $A$ as a central simple $\Z(A)$-algebra.

We say that $x\in A$ is \emph{symmetric} or \emph{skew-symmetric} (\emph{with respect to $\s$}) if $\s(x)=x$ or $\s(x)=-x$, respectively.
We let 
$$\Sym{A, \sigma}=\{ x\in A\mid \sigma(x)=x\}\quad\mbox{ and
}\quad\Skew{A, \sigma}=\{x\in A\mid \sigma(x)= - x\}\,.$$
These are $K$-linear subspaces of $A$.

Assume from now on that the characteristic of $K$ is different from $2$.
Then
\[A=\Sym{A,\s}\oplus \Skew{A,\s}\]
and there exists $\ve\in\{-1,0,+1\}$ such that
 \[\dim_K(\Sym{A,\s})=\half n(n+\ve)\,\, \mbox{ and }  
\dim_K(\Skew{A,\s})=\half n(n-\ve)\,.\]
If $(A,\s)$ is of the first kind, then  $\ve=\pm 1$, and we say that  $(A,\s)$ is \emph{orthogonal} if $\ve=1$ and \emph{symplectic} if $\ve=-1$.
If $(A,\s)$ is of the second kind, then  $\ve=0$, and we  say that $(A,\s)$ is \emph{unitary}.
The integer $\ve$ is called the \emph{type of $(A,\s)$} and denoted  $\type(A,\s)$.
We say that $(A,\s)$ is \emph{unitary of inner type} when $Z(A)\simeq F\times F$.
(The term is motivated by a corresponding notion for algebraic groups.)

Following \cite[\S 12]{BOI}, given a $K$-algebra with involution $(A,\s)$, we denote 
\begin{eqnarray*}
\Sim(A,\s) &=& \{x \in \mg{A} \mid \s(x)x \in \mg K\}  \quad\text{and}\\
\G(A,\s) &= &\{ \s(x)x \mid x \in \Sim(A,\s) \}\,;
\end{eqnarray*}
note that these are subgroups of $\mg A$ and of $\mg K$, respectively.

Let $\Psi=(A,\s)$ be a $K$-algebra with involution. 
In this notational setting we denote the underlying $K$-algebra $A$ by $\underline{\Psi}$.
We say that $\Psi$ is \emph{split} if $\underline{\Psi}$ is isomorphic to a matrix algebra over $\Z(\underline{\Psi})$.

For any field extension $L/K$ the $L$-algebra with involution $(A\ox_K L,\s\ox \id_L)$ is denoted by $\Psi_L$.
Note that $\type(\Psi_L)=\type(\Psi)$ and that $\underline{\Psi_L}$ is simple if and only if $\underline{\Psi}$ is simple and 
$L$ is linearly disjoint to $\Z(\underline{\Psi})$ over $K$.

We now consider two $K$-algebras with involution $\Psi=(A,\s)$ and $\Theta=(B,\vt)$.
A \emph{homomorphism of algebras with involution} $\Psi \too \Theta$ is a $K$-homomorphism $f:A\too B$ satisfying $\vt\circ f = f \circ \s$. 
A homomorphism  is called an \emph{embedding} if it is injective and an \emph{isomorphism} if it is bijective.
 We write $\Psi \simeq \Theta$ if there exists and isomorphism $\Psi\too\Theta$.
(This occurs if and only if either $\Psi$ and $\Theta$ are adjoint to two similar hermitian or skew-hermitian forms over some $K$-division algebra with involution or $\Psi$ and $\Theta$ are both unitary of inner type
with $\underline{\Psi}\simeq\underline{\Theta}$.)

Except when $\Psi$ and $\Theta$ are both unitary with different centres, we can define 
their tensor product $\Psi\ox \Theta$.
If $\Psi$ and $\Theta$ are not both unitary, let $\Psi\ox \Theta$ denote the $K$-algebra with involution $(A\ox_K B,\s\ox\vt)$.
If $\Psi$ and $\Theta$ are both unitary and with same centre $L$, then $\Psi\ox \Theta$ is the unitary $K$-algebra with involution $(A\ox_L B,\s\ox\vt)$, whose centre is also $L$.
In particular, for a positive integer $n$ 
the tensor power $\Psi^{\ox n}$ of a $K$-algebra with involution $\Psi$ is defined. 
Furthermore, in each of the cases where we defined $\Psi\ox \Theta$ we have 
\[\deg(\Psi\ox\Theta)=\deg(\Psi)\cdot\deg(\Theta)
\quad \mbox{ and }\quad 
\type(\Psi\ox\Theta)=\type(\Psi)\cdot\type(\Theta) \,.\]

Let $\vf=(V,B)$ be a quadratic form over $K$. Consider the split central simple $K$-algebra $\End_K(V)$.
Let $\s:\End_K(V)\too \End_K(V)$ denote the involution determined 
by the formula
\[\qquad B(f(u), v) = B(u, \s(f)(v))\quad\text{ for all } u, v\in V\text{ and } f\in \End_K(V).\]
We  denote this involution $\s$ by $\ad_B$ and call it the \emph{adjoint involution of $\vf$}. Furthermore, we call $(\End_K(V),\ad_B)$ the \emph{adjoint algebra with involution of $\vf$} and denote  it by $\Ad(\vf)$. Note that $\Ad(\vf)$ is split orthogonal and determines $\vf$ up to similarity.

\begin{ex}
Let $n$ be a positive integer and $\vf=n\times \la 1\ra$, the $n$-dimensional form $\la 1,\dots,1\ra$ over $K$.
Then $\Ad(\vf)\simeq (\M_n(K), \tp)$ where $\M_n(K)$ is the $K$-algebra of $n\!\times\! n$-matrices over $K$ and $\tp$ is the transpose involution.
\end{ex}

\begin{prop}\label{P:ad-mult}
For quadratic forms $\vf$ and $\psi$ over $K$ we have $$\Ad(\vf\ox\psi)\simeq \Ad(\vf)\ox\Ad(\psi)\,.$$
\end{prop}

\begin{proof} Denoting $V$ and $W$ the underlying vector spaces of $\vf$ and $\psi$, respectively, the natural $K$-algebra isomorphism
$\End_K(V)\ox_K \End_K(W) \too \End_K(V \ox_K W)$ yields the required identification for the adjoint involutions.
\end{proof}

For
 a finite-dimensional $K$-algebra $A$ we denote by $\Trd_A:A\too \Z(A)$ its reduced trace map (cf. \cite[p.~5 and p.~22]{BOI}).

A \emph{$K$-quaternion algebra} is a central simple $K$-algebra of degree $2$.
Given a $K$-quaternion algebra $Q$, the map $\s:Q\too K, x\mapstoo \Trd_Q(x)-x$ is a $K$-involution, called the \emph{canonical involution on $Q$} and denoted by $\can_Q$; this is the unique symplectic $K$-involution on $Q$.
If $L$ is a quadratic \'etale extension of $K$ we denote by $\can_{L/K}$ the unique nontrivial $K$-automorphism of $L$.
We further set $\can_K=\id_K$. We call $(A,\s)$ a \emph{$K$-algebra with canonical involution} if it is of one of the forms $(K,\id_K)$, $(L,\can_{L/K})$ for a quadratic \'etale extension $L/K$, or $(Q,\can_Q)$ for a $K$-quaternion algebra $Q$.

\begin{prop}\label{P:can-invol}
Let $(A,\s)$ be a $K$-algebra with involution.
We have that $\Sym{A,\s}=K$ if and only if $(A,\s)$ is a $K$-algebra with canonical involution.
\end{prop}

\begin{proof}
If $\Sym{A,\s}=K$ then $\dim_K(A)=2^{1-\ve}$ for $\ve=\type(A,\s)$, so that $A$ is either $K$, a quadratic \'etale extension of~$K$, or a $K$-quaternion algebra.
In any of these three cases, the canonical involution of $A$ is the unique $K$-involution of type $\ve$, hence it is therefore equal to $\s$ and $\Sym{A,\s}=K$.
\end{proof}

The following statement goes back to Jacobson \cite{Jac40}.

\begin{prop}\label{P:JAC}
Let $\Psi$ be a  $K$-algebra with involution. 
Then $\Psi\simeq \Ad(\vf)\ox \Phi$ for a $K$-algebra with canonical involution~$\Phi$ and a quadratic form $\vf$ over~$K$ if and only if $\Psi$ is either split or symplectic of index $2$.
\end{prop}

\begin{proof}
Clearly, any $K$-algebra with canonical involution $\Phi$ is 
either split or symplectic of index $2$, and thus so is $\Phi\otimes \Ad(\vf)$ for any quadratic form $\vf$ over $K$.
Assume now that $\Psi$ is either split or symplectic of index $2$.
Then  $\Psi$ is adjoint to a hermitian form over a $K$-algebra with canonical involution $\Phi$.
Since such a form has a diagonalisation with entries in $\Sym{\Psi}=K$,
we obtain that $\Psi\simeq \Ad(\vf)\ox\Phi$ for a form $\vf$ over $K$.
 \end{proof}

For computational purposes we augment the classical notation for  quaternion algebras
in terms of pairs of field elements to take into account an involution. Let $a,b\in \mg K$ and let $Q$ be the $K$-algebra with basis $(1,i,j,k)$, where $i^2=a$, $j^2=b$ and $ij=-ji=k$. This quaternion algebra is denoted by $(a,b)_K$. For $\delta, \ve \in \{+1,-1\}$ there is a unique $K$-involution $\s$ on $Q$ such that $\s(i)=\delta i$ and $\s(j)=\ve j$. We denote the pair $(Q,\s)$ by  
\begin{eqnarray*}
(a\qi b)_K & \text{ if } & \delta=+1,\ \ve=+1,\\
(a\qil b)_K &  \text{if} & \delta=-1,\  \ve=+1,\\
(a\qir b)_K &  \text{if} & \delta=+1,\  \ve=-1,\\
(a\qilr b)_K & \text{if} & \delta=-1,\ \ve=-1.
\end{eqnarray*}
In particular,  $(a\qilr b)_K$ denotes the quaternion algebra $(a,b)_K$ together with its canonical involution. 
Any $K$-quaternion algebra with orthogonal involution is isomorphic to $(a \qil b)_K$ for some $a,b \in \mg K$.
Note that $(a \qi b)_K \simeq (-ab \qil b)_K$ and $(a\qir b)_K \simeq (b \qil a)_K$ for any $a,b \in \mg K$. 
Hence the presence or absence of a dot indicates on which quadratic \'etale $K$-subalgebra the involution restricts to the nontrivial automorphism or to the identity, respectively.

\begin{prop}\label{P:quat-qil-rep}
Let $Q$ be a $K$-quaternion algebra. Let $a,b\in\mg{K}$ be such that $Q\simeq (a,b)_K$. Then $(Q,\can_Q)\simeq (a\qilr b)_K$.
Moreover, for $i\in\mg{Q}\setminus\mg{K}$ with $i^2=a$ and $\tau=\Int(i)\circ\can_Q$ we have  $(Q,\tau)\simeq (a\qil b)_K$.
 \end{prop}
 
\begin{proof}
Since $\can_Q$ is the only symplectic involution on $Q$, any $K$-isomorphism $Q\too (a,b)_K$ is also an isomorphism of $K$-algebras with involution. Hence, $(Q,\can_Q) \simeq (a\qilr b)_K$.
We choose an element $i\in\mg{Q}\setminus\mg{K}$ such that $i^2=a$.
Then $V=\{j\in Q\mid ij+ji=0\}$ is the orthogonal complement of $K[i]$ in $Q$ with respect to
the symmetric $K$-bilinear form $B: Q\x Q\too K, (x,y)\mapstoo \frac 12\Trd_Q(\can_Q(x)\cdot y)$.
By \cite[Chap.~2, (11.4)]{Scharlau}
we obtain that
$(Q,B)\simeq \lla a,b\rra$. Since
$(K[i], B|_{K[i]})\simeq \lla a\rra$ it follows that 
$(V, B|_{V})\simeq -b\lla a\rra$.
As for any $j\in V$ we have $B(j,j)=-j^2$, there exists an element $j\in V$ with $j^2=b$.
For $\tau=\Int(i)\circ\can_Q$ we obtain that $\tau(i)=-i$ and $\tau(j)=j$, whereby $(Q,\tau)\simeq (a\qil b)_K$.
\end{proof}

For $a\in \mg K$, let $(a)_K$ denote the unitary $K$-algebra with canonical involution $(L,\can_{L/K})$ where
$L = K[X]/(X^2-a)$; it is of inner type 
if and only if $a\in\sq{K}$.

\begin{cor}\label{C:quat-awi-class}
Let $\Phi$ be a $K$-quaternion algebra with involution.
If $\Phi$ is orthogonal there exist $a,b\in\mg{K}$ such that $\Phi\simeq (a\qil b)_K$.
If $\Phi$ is symplectic there exist $a,b\in\mg{K}$ such that $\Phi\simeq (a\qilr b)_K$.
If $\Phi$ is unitary, there exist $a,b,c\in\mg{K}$ such that $\Phi\simeq (a\qilr b)_K\ox (c)_K$.
\end{cor}

\begin{proof}
Assume that $\Phi$ is of the first kind and let $\Phi=(Q,\s)$.
If $\Phi$ is symplectic, then $\s=\can_Q$ and we choose $a,b\in\mg{Q}$ such that $Q\simeq (a,b)_K$ to obtain by \Cref{P:quat-qil-rep} that $\Phi\simeq (a\qilr b)_K$.
If $\Phi$ is orthogonal, we choose $i\in\Skew{Q,\s}\cap\mg{Q}$ and $a,b\in\mg{K}$ with   $a=i^2$  and $Q\simeq (a,b)_K$, and obtain that $\s=\Int(i)\circ\can_Q$, so that $\Phi\simeq (a\qil b)_K$ by \Cref{P:quat-qil-rep}.

Assume now that $\Phi$ is of the second kind. 
From \cite[(2.22)]{BOI} we obtain that $\Phi\simeq (Q,\can_Q)\ox (c)_K$ for a $K$-quaternion algebra $Q$ and an element $c\in\mg{K}$, and by the above there exist $a,b\in\mg{K}$ such that $(Q,\can_Q)\simeq (a\qilr b)_K$.
\end{proof}

\begin{prop}\label{P:-qil-characterisation}
Let $a,b,c,d \in \mg K$.
 We have  $(a \qil b)_K \simeq (c \qil d)_K$ if and only if $a \sq K =c \sq K$ and $bd \in \D_K\lla a\rra$.
 \end{prop}
 
\begin{proof}
We set $(Q,\tau)=(a\qil b)_K$.
There exist $i\in\Skew{Q,\tau}$ and $j\in\Sym{Q,\tau}$ with $i^2=a$, $j^2=b$, and $ij+ji=0$.

Assuming that $bd\in \D_K\lla a\rra$, we may write $d=b(u^2-av^2)$ with $u,v\in K$ and obtain for $g=uj+vij$ that $g\in \Sym{Q,\tau}$, $gi+ig=0$, and $g^2=d$.
If further $c\sq{K}=a\sq{K}$, then $c=f^2$ for some $f\in i\mg{K}$, and we have $f\in\Skew{Q,\tau}$ and $gf+fg=0$, and conclude that $(Q,\tau)\simeq  (c\qil d)_K$.

For the converse, suppose that
$(Q,\tau)\simeq  (c\qil d)_K$. There exist $f\in\Skew{Q,\tau}$ and $g\in\Sym{Q,\tau}$ with $f^2=c$, $g^2=d$, and $fg+gf=0$.
It follows that $i K=\Skew{Q,\tau}=f K$, so that $a\sq{K}=c\sq{K}$. Moreover, $ig+gi=0$ and $jgi=ijg$.
As $K[i]$ is a maximal commutative $K$-subalgebra of $Q$, we obtain that $jg\in K[i]$.
Writing $jg=x+iy$ with $x,y\in K$, we obtain that $$bd=j^2g^2=j(x+iy)g=(x-iy)jg=(x-iy)(x+iy)=x^2-ay^2\,,$$ whence 
$bd\in D_K\lla a\rra$.
\end{proof}

\begin{rem}
For a $K$-algebra with orthogonal involution $(A,\s)$ of degree $2m$ the class in $\scg{K}$ given by the reduced norm of an arbitrary element $u\in\mg{A}\cap\Skew{A,\s}$ times $(-1)^m$ yields an invariant of $(A,\s)$, called the \emph{discriminant} (cf.~\cite[(7.1) and (7.2)]{BOI}).
For $a,b\in\mg{K}$ the discriminant of $(a\qil b)_K$ is $a\sq{K}$.
Hence \Cref{P:-qil-characterisation} contains the observation that this is an invariant in the case where $A$ is a $K$-quaternion algebra.
\end{rem}

\begin{prop}\label{P:flipflop}
For $a,b,c,d \in \mg K$ we have 
\begin{align*}
(a\qil b)_K\ox (c\qil d)_K & \simeq  (a\qilr bc)_K\ox (c\qilr ad)_K \quad \textrm{and} \\
(a\qil b)_K\ox (c\qilr d)_K & \simeq  (a\qilr bc)_K\ox (c\qil ad)_K\,.
\end{align*}
\end{prop}

\begin{proof}
Let $(A,\s)=(a\qil b)_K\ox (c\qil d)_K$.
Then there exist elements $i,j,f,g\in\mg{A}$ such that $\s(i)=-i$, $\s(j)=j$, $\s(f)=-f$, $\s(g)=g$, $i^2=a$, $j^2=b$, $f^2=c$, $g^2=d$, $ij+ji=fg+gf=0$, and each of $i$ and $j$ commutes with each of $f$ and $g$.
Set $j'=fj$ and $g'=ig$. Then
$\s(i)=-i$, $\s(j')=-j'$, $\s(f)=-f$, $\s(g')=-g'$, $i^2=a$, $j'^2=bc$, $f^2=c$, $g'^2=ad$, $ij'+j'i=fg'+g'f=0$, and each of $i$ and $j'$ commutes with each of $f$ and $g'$.
The $K$-subalgebra $Q$ of $A$ generated by $i$ and $j'$ commutes elementwise with the $K$-subalgebra $Q'$ of $A$ generated by $f$ and $g'$, and $Q$ and $Q'$ are $\s$-stable.
Hence 
$$(A,\s)\simeq (Q,\s|_Q)\ox (Q',\s|_{Q'})\simeq (a\qilr bc)_K\ox (c\qilr ad)_K\,.$$
This shows the first isomorphism.
The proof of the second isomorphism is almost identical, with the only difference that $\s(g)=-g$ and $\s(g')=g'$.
\end{proof}

\begin{prop}
For $a,b \in \mg K$, we have
\[\G\,(a\qilr b)_K  = \D_K \lla a,b\rra \quad\text{ and }\quad
\G\,(a\qil b)_K = \D_K\lla a\rra\cup b\D_K\lla a\rra.\]
\end{prop}

\begin{proof}
Let $Q=(a,b)_K$ and $u,v\in \mg{Q}$ with $u^2=a$, $v^2=b$ and $uv+vu=0$. Then $\Sim(Q,\can_Q)=\mg{Q}$ and thus $\G(Q,\can_Q)=\D_K\lla a,b\rra$.
For $\tau=\Int(u)\circ\can_Q$ we obtain $\Sim(Q,\tau)=\mg{K(u)}\cup v\mg{K(u)}$ and thus $\G(Q,\tau)=\D_K\lla a\rra\cup b\D_K\lla a\rra$.
\end{proof}


A $K$-algebra with involution $(A,\s)$ with centre $L$ gives rise to a regular hermitian form $T_{(A,\s)}:A\x A\too L$ over  $(L,\can_{L/K})$ defined by $T_{(A,\s)}(x,y)=\Trd_A(\s(x)y)$; this follows from \cite[(2.2) and (2.16)]{BOI}.
We further obtain a regular symmetric $K$-bilinear form $T_\s:A\x A \too K$ defined by $T_\s(x,y)=\half \Trd_A (\s(x)y+\s(y)x)$.
 Note that if $L=K$ then $T_\s=T_{(A,\s)}$, otherwise $2T_\s=T\circ T_{(A,\s)}$ where $T$ is the trace of ${L/K}$.

Given a $K$-algebra with involution $\Psi=(A,\s)$, we denote by $\itr(\Psi)$ the quadratic form $(A,T_\s)$ over $K$.
Note that $\dim(\itr(\Psi))=\dim_K(A)$. If $\Psi$ is a $K$-algebra with canonical involution, then $\itr(\Psi)$ is the norm form of the $K$-algebra $\unl{\Psi}$.

Note that for a quadratic form $\vf$ the notation $2\vf$ refers to $\la 2\ra\ox\vf$ (the form obtained by scaling $\vf$ by $2$), which needs to be distinguished from $2\times\vf=\vf\perp \vf$.

\begin{ex}\label{E:itr}
For $a\in\mg{K}$ we have $\itr\,(a)_K=\lla a\rra$. For $a,b\in\mg{K}$ we have $\itr\,(a\qil b)_K=2\lla a,-b\rra$ and $\itr\,(a\qilr b)_K=2\lla a,b\rra$. 
\end{ex}

\begin{prop}\label{P:itr-prod}
Let $\Psi$ and $\Theta$ be $K$-algebras with involution.
If  $\Psi$  is of the first kind, then
$\itr(\Psi\ox\Theta)\simeq  \itr(\Psi)\ox\itr(\Theta)$.
If $\Psi$ and $\Theta$ are both unitary with same centre, then
$2\times\itr(\Psi\ox\Theta)\simeq  \itr(\Psi)\ox\itr(\Theta)$. 
\end{prop}

\begin{proof}
Write $\Psi=(A,\s)$ and $\Theta=(B,\tau)$. 
Let $K'=\Z(A)$ and $L=\Z(B)$.
In view of the claims we may assume that $K'\subseteq L$.
For $a\in A$ and $b\in B$ we have $\Trd_{A\ox_{K'}B}(a\ox b)=\Trd_A(a)\cdot \Trd_B(b)$, as one verifies by reduction to the split case.
Hence, 
 $T_{\Psi\ox \Theta}$ and $T_{\Psi}\ox T_{\Theta}$ coincide as hermitian forms on  $A\ox_{K'} B$ with respect to $(L,\can_{L/K})$.
If $L=K$ then we are done.
Assume now that $(L,\can_{L/K})\simeq (c)_K$  where $c\in\mg{K}$.
Then $\itr(\Theta)\simeq \lla c\rra\ox\vt$ for a form $\vt$ over $K$.
If now $K'=K$ then $ \itr(\Psi)\ox\itr(\Theta)\simeq \lla c\rra\ox(\itr(\Psi)\ox \vt)\simeq \itr(\Psi\ox \Theta)$.
In the remaining case $K'=L$ and $\itr(\Psi)\simeq \lla c\rra\ox\psi$ for a quadratic form $\psi$ over $K$, and we obtain that
$\itr(\Psi)\ox\itr(\Theta)\simeq \lla c,c\rra\ox \psi\ox\vt\simeq 2\times \lla c\rra\ox(\psi\ox\vt)\simeq 2\times
\itr(\Psi\ox \Theta)$.
\end{proof}

\begin{prop}
\label{P:trace-qf-square}
For any  form $\vf$ over $K$ we have $\itr(\Ad(\vf))\simeq \vf\ox\vf$.
\end{prop}

\begin{proof}
 See \cite[(11.4)]{BOI}.
\end{proof}

Let $A$ be a finite-dimensional $K$-algebra.
For $a\in A$ let $\lambda_a\in \End_K(A)$ be given by $\lambda_a(x)=ax$
for $x\in A$. The $K$-algebra homomorphism $\lambda:A\too \End_K(A)$, $a \mapstoo \lambda_a$ thus obtained  is called the
\emph{left regular representation of $A$}.

\begin{prop}\label{P:awi-embed} 
Let $\Psi$ be a $K$-algebra with involution. The left regular representation of  $\underline{\Psi}$ yields an embedding of $\Psi$  into  $\Ad(\itr(\Psi))$.
\end{prop}

\begin{proof}  Write $\Psi=(A,\s)$. For $a, x, y\in A$ we have 
$T_\s(x, \lambda_{\s(a)}(y))=T_\s(\lambda_a(x),y)$. 
Thus $\lambda$ identifies $\s$ with the restriction to $\lambda(A)$ of the involution adjoint to $T_\s$.
\end{proof}

\begin{prop}
\label{P:awi-square}
Let $\Psi$ be a $K$-algebra with involution of the first kind.
Then
 $$\Psi\ox\Psi \simeq \Ad(\itr(\Psi))\,.$$
\end{prop}

\begin{proof} We expand the proof of \cite[(11.1)]{BOI}. Let $\Psi=(A,\s)$. 
Consider the $K$-algebra homomorphism 
$\s_*: A\ox_K A\too \End_K(A)$ determined by $\s_*(a\ox b)(x)=ax\s(b)$ for all $a,b,x \in A$. 
As $A\ox_K A$ is simple and of the same dimension as $\End_K(A)$,  $\s_*$ is an isomorphism.  For $a,b,x,y\in A$ we have 
$$T_\s(x, \s_*(\s(a)\ox \s(b))(y))=
T_\s (\s_*(a\ox b)(x), y)\,.$$ Thus $\s_*$ identifies 
the involution $\s\ox\s$  with the adjoint involution of $T_\s$.
\end{proof}

\section{Hyperbolicity} 
\label{sec6}

Following \cite[(2.1)]{BST}, we say that the $K$-algebra with involution $(A,\s)$ is \emph{hyperbolic} if there exists an element $e\in A$ with $e^2=e$ and $\s(e)=1-e$.
If $(A,\s)$ is adjoint to a hermitian form over a $K$-division algebra with involution, then it is hyperbolic if and only if the hermitian form is hyperbolic.

\begin{prop}\label{P:hyper-char}
The $K$-algebra with involution $(A,\s)$ is hyperbolic if and only if there exists $f\in\Skew{A,\s}$ with $f^2=1$, if and only if $(1)_K$ embeds into $(A,\s)$.
\end{prop}

\begin{proof}
The second equivalence is obvious.
To prove the first equivalence, given  $e\in A$ with $e^2=e$ and $\s(e)=1-e$, we see that $f=2e-1$ satisfies $\s(f)=-f$ and $f^2=1$, and conversely, for $f\in A$ with these properties, the element $e=\frac{1}{2}(f-1)$ satisfies $e^2=1$ and $\s(e)=1-e$.
\end{proof}

\begin{cor}\label{C:trivial-hyper}
Let $\Psi$ be a $K$-algebra with involution which is either split symplectic or unitary of inner type.
Then $\Psi$ is hyperbolic. 
\end{cor}

\begin{proof}
Using \Cref{P:JAC} we have that $\Psi\simeq \Ad(\vf)\ox \Phi$ for a $K$-algebra with canonical involution $\Phi$, and conclude that $\Phi\simeq (1\qilr 1)_K$ or $\Phi\simeq (1)_K$.
In either case $\Psi$ contains $(1)_K$ and thus is hyperbolic by \Cref{P:hyper-char}.
\end{proof}

\begin{prop}
Let $\Psi$ and $\Theta$ be $K$-algebras with involution.
\label{C:hyper-trace}
\quad
\begin{enumerate}[$(a)$]
\item If $\Psi$ and $\Theta$ are hyperbolic,  $\type(\Phi)=\type(\Psi)$ and $\underline{\Psi}\simeq\underline{\Theta}$,
then $\Psi\simeq \Theta$. 
\item If $\Psi$ is hyperbolic, then $\Psi\ox\Theta$ is hyperbolic. 
\item If $\Psi$ is hyperbolic, then $\itr(\Psi)$ is hyperbolic. 
\end{enumerate}
\end{prop}

\begin{proof}
$(a)$ If $\Psi$ and $\Theta$ are  
unitary of inner type, the statement follows from \cite[(2.14)]{BOI}. 
Otherwise $\Psi$ and $\Theta$ are adjoint to hyperbolic hermitian or skew-hermitian forms of the same dimension over a common $K$-division algebra with involution, and these are necessarily isometric.

$(b)$ This is obvious.

$(c)$ By \Cref{P:awi-embed},  $\Psi$ embeds into $\Ad(\itr(\Psi))$. This yields the statement.
\end{proof}

\begin{prop}
\label{P:simhyp}
Let $\Psi$ be a $K$-algebra with involution and $a\in \mg{K}$. We have that  $a \in \G(\Psi)$ if and only if $\Ad\lla a\rra \ox \Psi$ is hyperbolic.
\end{prop}

\begin{proof} 
See \cite[(12.20)]{BOI}.
\end{proof}

\begin{thm}[Bayer-Fluckiger, Lenstra]\label{T:BFL}
Let $\Psi$ be a $K$-algebra with involution and
let $L/K$ be a finite field extension of odd degree.
Then $\Psi_L$ is hyperbolic if and only if  $\Psi$ is hyperbolic.
\end{thm} 

\begin{proof}
See \cite[(6.16)]{BOI}.
\end{proof}

The following is a reformulation of the main result in \cite{Jac40}.

\begin{thm}[Jacobson]\label{T:JAC} 
Let $\Phi$ be a $K$-algebra with canonical involution and $\vf$ a  quadratic form over $K$.
Then $\Ad(\vf)\ox \Phi$ is hyperbolic if and only if $\vf\ox\itr(\Phi)$ is hyperbolic.
\end{thm}

\begin{proof}
Let $\vf=(V,B)$ and $\Phi=(A,\s)$.  
We may assume that $A$ has no zero-divisors since otherwise $\Phi$ and $\itr(\Phi)$ are hyperbolic.
The $K$-algebra with involution $\Ad(\vf)\ox \Phi$ is adjoint to the hermitian form $(V_A,h)$ over $\Phi$ obtained from $\vf$, with $V_A=V\ox_KA$ and  $h:V_A\x V_A\rightarrow A$  determined by $h(v\ox a, w\ox b)=\s({a})B(v,w)b$ for $a,b\in A$ and $v,w\in V$.
Let $T:A\to K$, $x\mapsto x+\s(x)$.
Note that $T\circ h$ coincides with $B\otimes T_\s$ on $V_A\x V_A$ up to a scalar factor. The isotropic vectors of the hermitian form $h$ and of the quadratic forms $T\circ h$ and $B\otimes T_\s$ therefore coincide. In particular, a maximal totally isotropic $A$-subspace for $h$ is the same as a maximal totally isotropic $K$-subspace for $B\otimes T_\s$. This yields the statement.
 \end{proof}

\begin{thm}[Bayer-Fluckiger, Shapiro, Tignol]
\label{quExt} Let $\Psi$ be a $K$-algebra with involution and 
 $a\in\mg{K}\setminus\sq{K}$.  Then
$\Psi_{K(\sqrt{a})}$ is hyperbolic if and only if $(a)_K$ embeds into $\Psi$ or  
$\Psi \simeq \Ad(\varphi)$ for a quadratic form $\vf$ over $K$ whose anisotropic part is a multiple of $\lla a\rra$.
\end{thm}
\begin{proof}
Assume first that $\Psi$ is split orthogonal, so that $\Psi\simeq \Ad(\vf)$ for a form $\vf$ over $K$. Then $\Psi_{K(\sqrt{a})}$ is hyperbolic if and only if $\vf_{K(\sqrt{a})}$ is hyperbolic, which by \cite[Chap.~VII, (3.2)]{Lam} is the case if and only if the anisotropic part of $\vf$ is a multiple of $\lla a\rra$.
In the remaining cases, the statement is proven in \cite[(3.3)]{BST} for involutions of the first kind, and an  adaptation of the argument for involutions of the second kind is provided in \cite[(3.6)]{LU}.
\end{proof}

\begin{rem}  
There is an overlap in the two cases of the characterization given in \Cref{quExt}.
Assume that $a\in\mg{K}\setminus\sq{K}$ and $\vf$ is a form over $K$.
Then $(a)_K$ embeds into $\Ad(\vf)$ if and only if 
$\vf$ is a multiple of $\lla a\rra$, if and only if the anisotropic part $\vf_\an$ of $\vf$ is a multiple of $\lla a\rra$ and $\vf\simeq \vf_\an\perp 2m\times \hh$ for some $m\in\nat$. 
\end{rem}

\begin{cor}\label{C:DinG} 
Let $\Psi$ be a $K$-algebra with involution. 
For any $a\in\mg{K}\setminus\sq{K}$ such that $\Psi_{K(\sqrt{a})}$ is hyperbolic, we have that $\D_K\lla a\rra \subseteq \G(\Psi)$.
\end{cor}

\begin{proof}  
As $\D_K\lla a\rra=\G((a)_K)$, the statement follows immediately  from \Cref{quExt}.
\end{proof}

The following was already observed in \cite[p.~91, Examples~(b) and (c)]{BOI}.

\begin{prop}
\label{P:bqhyp}
Let $Q_1$ and $Q_2$ be $K$-quaternion algebras. The $K$-algebra with involution $(Q_1, \can_{Q_1})\ox (Q_2, \can_{Q_2})$ is hyperbolic if and only if one of $Q_1$ and $Q_2$ is split.
\end{prop}

\begin{proof}
Let $(A,\s)= (Q_1, \can_{Q_1})\ox (Q_2, \can_{Q_2})$.
If one of the factors is split, it is hyperbolic, and thus $(A,\s)$ is hyperbolic.
Assume now that $(A,\s)$ is hyperbolic.
Then by \Cref{P:hyper-char} there exists $f\in\Skew{\s}$ with $f^2=1$.
We identify $Q_1$ and $Q_2$ with $K$-subalgebras of $A$ that commute with each other elementwise and such that $\s|_{Q_i}=\can_{Q_i}$ for $i=1,2$.
Then $\Skew{\s}=Q_1'\oplus Q_2'$ where $Q_i'$ is the $K$-subspace of pure quaternions of $Q_i$ for $i=1,2$.
Writing $f=f_1+f_2$ with $f_i\in Q_i'$ for $i=1,2$, we obtain that
 $1=f^2=f_1^2+f_2^2+2f_1f_2$. As $f_1^2, f_2^2\in K$, we conclude that $f_1f_2\in K$. 
 This is only possible if $f_1f_2=0$, that is, if either $f_1=0$ or $f_2=0$.
If, say, $f_2=0$, then $f=f_1$, which then is a hyperbolic element with respect to $\s$ contained in $Q_1$,
whereby $Q_1$ is split.
Hence, one of $Q_1$ and $Q_2$ is split.
 \end{proof}

\begin{thm}[Karpenko, Tignol]
\label{T:KT}
Let  $\Psi$ be a $K$-algebra with involution such that $\Psi\ox \Psi$ is split.
If $\Psi$ is not hyperbolic, then there exists a field extension $L/K$ such that $\Psi_L$ is not hyperbolic and, either $\Psi_L$ is split, or $\Psi$ is symplectic and $\Psi_L$ is Brauer equivalent to an $L$-quaternion algebra.
\end{thm}
\begin{proof}
See \cite[(1.1)]{Kar} for the orthogonal case and \cite[(A.1) and (A.2)]{Tignol} for the other cases.
\end{proof}

Note that the condition in \Cref{T:KT} that $\Psi\ox\Psi$ be split is trivially satisfied if $\Psi$ is a $K$-algebra with involution of the first kind.

We mention separately  the following special case of \Cref{T:KT}, which was obtained earlier by  more classical methods. 
It will be used in \Cref{P:final}.

\begin{thm}[Dejaiffe, Parimala, Sridharan, Suresh]\label{T:DPSS}
Let $a,b\in\mg{K}$ and let $L$ be the function field of the conic $aX^2+bY^2=1$ over $K$.
Let $\Psi$  be a $K$-algebra with orthogonal involution such that $\underline{\Psi}\sim (a,b)_K$.
Then $\Psi$ is hyperbolic if and only if $\Psi_L$ is hyperbolic.
\end{thm}
\begin{proof}
If the conic $aX^2+bY^2=1$ is split over $K$, then $L$ is a rational function field over $K$ and the statement is obvious.
Otherwise $\Phi=(a\qilr b)_K$ is a $K$-quaternion division algebra with involution and $\Psi$ is adjoint to a skew-hermitian form over $\Phi$, in which case the statement follows alternatively from \cite{Dej} or \cite[(3.3)]{PSS}.
\end{proof}

\section{Algebras with involution over real closed fields} 
\label{sec7}

Let $\Psi$ be a $K$-algebra with involution.
For $n\geq 1$ we set $n\times \Psi=\Ad(n\times \la 1\ra)\ox\Psi$.

\begin{prop}\label{T:pyth-hamilton-orth}
Assume  that $K$ is pythagorean and that $\Psi$ is orthogonal and such that $\underline{\Psi}\sim  (-1,-1)_K$. Then $\Psi \simeq \Ad(\vf) \ox (-1 \qil -1)_K$ for  a form $\vf$ over $K$.
Moreover, $2\times \Psi$ is hyperbolic.
\end{prop}

\begin{proof}
Let $Q=(-1,-1)_K$.
We may identify $\Psi$ with $(\End_Q(V),\s)$ where $V$ is a finite-dimensional right $Q$-vector space and $\s$ is the involution adjoint to a regular skew-hermitian form $h:V\x V\lra Q$ with respect to $\can_Q$.
Since $K$ is pythagorean, any maximal subfield of $Q$ is $K$-isomorphic to $K(\sqrt{-1})$.
We fix a pure quaternion $i\in Q$ with $i^2=-1$ and obtain that any invertible pure quaternion in $Q$ is conjugate to an element of $i\mg{K}$.
This yields that $h$ has a diagonalization with entries in $i\mg{K}$.
It follows that $i h:V\x V\lra Q$ is a hermitian form with respect to the involution $\tau=\Int(i)\circ \can_Q$ and has a diagonalization with entries in $\mg{K}$.
This yields that $\Psi\simeq \Ad(\vf)\ox (Q,\tau)$ for a form $\vf$ over $K$.
Moreover, $(Q,\tau)\simeq (-1\qil -1)_K$.
This shows the first claim.

As $\Ad\la 1,1\ra\simeq (-1\qil 1)_K$ we obtain using \Cref{P:flipflop} that 
$$2\times (Q,\tau)\simeq (-1\qil 1)_K\otimes (-1\qil -1)_K\simeq (-1\qilr {-1})_K\otimes (-1\qilr 1)_K\,.$$
By \Cref{P:hyper-char} this $K$-algebra with involution is hyperbolic, and thus so is $2\times \Psi$.
\end{proof}

\begin{thm}
\label{P:rcai}
Assume that $K$ is real closed.
\begin{enumerate}[$(a)$]
\item If $\Psi$ is split orthogonal, then $\Psi \simeq \Ad(r\times\la 1\ra\perp\eta)$ for a hyperbolic form $\eta$ over $K$ and  $r\in\nat$ such that $\sign(\itr(\Psi))=r^2$.
\item 
If $\Psi$ is non-split orthogonal, then $\Psi\simeq r\times (-1\qil -1)_K$ for a positive integer $r$, the form $\itr(\Psi)$ is hyperbolic, and $\Psi$ is hyperbolic if and only if $r$ is even.
\item If $\Psi$ is split symplectic, then $\Psi\simeq r\times (1\qilr 1)_K$ for a positive integer $r$, and both $\Psi$ and $\itr(\Psi)$ are hyperbolic.
\item If $\Psi$ is non-split symplectic, then $\Psi \simeq \Ad(r\times\la 1\ra\perp\eta)\otimes (-1\qilr -1)_K$ for a hyperbolic form $\eta$ over $K$ and  $r\in\nat$ such that $\sign(\itr(\Psi))=4r^2$.
\item If $\Psi$ is  
unitary and $\underline{\Psi}$ is simple, then  $\Psi \simeq \Ad(r\times\la 1\ra\perp\eta)\otimes (-1)_K$ for a hyperbolic form $\eta$ over $K$ and  $r\in\nat$ such that $\sign(\itr(\Psi))=2r^2$.
\item If $\Psi$ is unitary of inner type, then $\Psi\simeq r\times (1)_K$ for a positive integer, and both $\Psi$ and $\itr(\Psi)$ are hyperbolic.
\end{enumerate}
These cases are mutually exclusive and cover all possibilities, and the integer $r$ is unique in each case.
\end{thm}

\begin{proof} 
It is clear that exactly one of the conditions in $(a)$--$(f)$ is satisfied.
As $K$ is real closed, the only finite-dimensional $K$-division algebras are $K$, $K(\sqrt{-1})$, and $(-1,-1)_K$.
We set
$$\Phi = \begin{cases} 
(K,\id_K) & \text{if $\Psi$ is split orthogonal,}\\
(-1\qil -1)_K & \text{if $\Psi$ is non-split orthogonal,}\\
(1 \qilr 1)_K & \text{if $\Psi$ is split symplectic,}\\
(-1\qilr-1)_K & \text{if $\Psi$ is non-split symplectic,}\\
(-1)_K  & \text{if $\Psi$ is  unitary and $\underline{\Psi}$ is simple,}\\
(1)_K  & \text{if $\Psi$ is unitary of inner type.} 
\end{cases}$$
Note that $\type(\Phi)=\type(\Psi)$ and either $\Phi$ and $\Psi$ are both 
unitary of inner type or the algebras 
 $\underline{\Phi}$ and $\underline{\Psi}$ are both simple with the same centre and Brauer equivalent.

Suppose that $\Psi$ is split-symplectic or  
unitary of inner type. Then
$\Psi$ is hyperbolic by \Cref{C:trivial-hyper}.
By \Cref{C:hyper-trace} it follows that $\itr(\Psi)$ is hyperbolic and that $\Psi\simeq r\times\Phi$ for some $r\in\nat$. This shows $(c)$ and $(f)$.

Next, suppose that $\Psi$ is non-split orthogonal.
Then by \Cref{T:pyth-hamilton-orth} we have that $\Psi\simeq \Ad(\vf)\otimes \Phi$ for a form $\vf$ over $K$, and as $\G\,(-1\qil -1)_K=\sq{K}\cup-\sq{K}=\mg{K}$ we may choose $\varphi$ to be $r\times\la 1\ra$ for some $r\in\nat$.
 We thus have 
$\Psi\simeq r\times \Phi$ with  $r\in\nat$ such that $\deg(\Psi)=2r$.
With this follows from \Cref{T:pyth-hamilton-orth} that $\Psi$ is hyperbolic if and only if $r$ is even.
Furthermore,  $\itr(\Phi)$ is hyperbolic by \Cref{E:itr}, and thus so is $\itr(\Psi)\simeq r^2\x \itr(\Phi)$.
This shows $(b)$.

In each of the remaining cases $(a)$, $(d)$, and $(e)$,  by \Cref{P:JAC} we have that $\Psi\simeq  \Ad(\vf)\ox \Phi$ for a form $\vf$ over $K$. Since $K$ is real closed and $\vf$ is determined up to a scalar factor, we choose $\vf$ to be $r\times \la 1\ra\perp\eta$ for some $r\in\nat$ and  a hyperbolic form $\eta$ over $K$.
It further follows that $\itr(\Psi)\simeq \vf\ox\vf\ox \itr(\Phi)$
by \Cref{P:itr-prod} and \Cref{P:trace-qf-square}
 and thus $\sign(\itr(\Psi))= r^2\cdot\sign(\itr(\Phi))$.
As in either case $\itr(\Phi)$ is  positive definite by \Cref{E:itr}, we have that 
$\sign(\itr(\Phi)) = \dim_K(\Phi)$.
This establishes the cases $(a)$, $(d)$, and $(e)$.

Finally, note that in each case the non-negative integer $r$ is determined by $\deg(\Psi)$ or by $\dim(\itr(\Psi))$, respectively.
\end{proof}

\begin{cor}
\label{C:rc}
Assume   $K$ is real closed. Then $\itr(\Psi)$ is hyperbolic if and only if $2\x \Psi$ is hyperbolic, if and only if either $\Psi$ is hyperbolic or $\Psi \simeq r\x (-1 \qil -1)_K$ where $r\in\nat$ is odd.
\end{cor}
\begin{proof}
We shall refer to the cases in \Cref{P:rcai}.
In each of the cases $(b)$, $(c)$, or $(f)$, both $\itr(\Psi)$ and  $2\x \Psi$ are hyperbolic.
Assume  that we are in one of the cases $(a)$, $(d)$, or $(e)$, and let $r$ be the integer occurring in the statement  for that case.
Then $\itr(\Psi)$ is hyperbolic if and only if $r=0$,  if and only if $\Psi$ is hyperbolic.
 \end{proof}

\begin{cor}\label{C:sign-def}
Let $P$ be an ordering of $K$ and $\Psi$ a $K$-algebra with involution. 
Then $\sign_P(\itr(\Psi))=[\Z(\underline{\Psi}):K]\cdot s^2$ for some $s\in\nat$.
\end{cor}
\begin{proof}
By \Cref{P:rcai} the statement holds in the case where $K$ is real closed and $P$ is the unique ordering of $K$.
The general case follows immediately by extending scalars to the real closure of $K$ at $P$.
\end{proof}

Let $P$ be an ordering of $K$.
The integer $s$ occurring in \Cref{C:sign-def} is called the
\emph{signature of $\Psi$ at $P$} and denoted $\sign_P(\Psi)$. 
With $k=[\Z(\underline{\Psi}):K]$ we thus have 
$$\sign_P(\Psi) =\sqrt{\mbox{$\frac{1}{k}$}\sign_P (\itr(\Psi))}\,.$$

\begin{prop}\label{P:sign-mult}
Let $\Psi$ and $\Theta$ be two $K$-algebras with involution.
If $\Psi$ and $\Theta$ are both unitary, assume that they have the same centre.
For every ordering $P$ of $K$ we have that $\sign_P(\Psi\ox\Theta)=\sign_P(\Psi)\cdot\sign_P(\Theta)$.
\end{prop}
\begin{proof}
This follows immediately from \Cref{P:itr-prod}.
\end{proof}

\begin{prop}
Let $\vf$ be a quadratic form over $K$. For every ordering $P$ of $K$ we have that $\sign_P (\Ad(\vf)) = | \sign_P (\vf)|$.
\end{prop}
\begin{proof}
This is clear as $\itr(\Ad(\vf))\simeq \vf\ox\vf$ by \Cref{P:trace-qf-square}.
\end{proof}

\section{Local-global principle for weak hyperbolicity} 
\label{sec8}

Let $\Psi$ be a $K$-algebra with involution. 
We call $\Psi$ \emph{weakly hyperbolic}
if there exists a positive integer $n$ such that $n\x \Psi$ is  hyperbolic.
We say that $\Psi$ \emph{has trivial  signature} and write $\sign(\Psi)=0$ to indicate that $\sign_P(\Psi)=0$ for every $P\in X_K$.

\begin{lem}\label{C:pmqExt}
Assume that there exists $a\in \mg K\setminus \pm\sq K$ such that $\Psi_{K(\sqrt{a})}$ and  $\Psi_{K(\sqrt{-a})}$ are  hyperbolic. Then $2\x \Psi$ is hyperbolic.
\end{lem}

\begin{proof} 
Let $a\in\mg{K}\setminus\pm\sq{K}$ be such that $\Psi_{K(\sqrt{a})}$ and  $\Psi_{K(\sqrt{-a})}$ are  hyperbolic.
By \Cref{C:DinG}  $a$ and $-a$ both belong to $\G(\Psi)$.
As $\G(\Psi)$ is a group, we conclude that $-1\in \G(\Psi)$, so  $2\x \Psi\simeq \Ad\lla -1\rra \ox \Psi $ is hyperbolic by \Cref{P:simhyp}.
\end{proof}

\begin{prop} 
\label{P:lev}
Assume that $K$ is nonreal and let $n\in\N$ be such that $-1$ is a sum of $2^n$ squares in $K$.
Then $2^{n+1}\x \Psi$ is hyperbolic.
\end{prop}

\begin{proof} By the assumption, the Pfister form $\pi=2^{n+1}\x \la1\ra$ over $K$ is isotropic, whereby it is hyperbolic.
Hence, $2^{n+1}\x \Psi \simeq \Ad(\pi) \ox \Psi$ is hyperbolic.
\end{proof}

\begin{lem}
\label{L:zk}
Assume that $2^n\x \Psi$ is not hyperbolic for any $n\in\N$, and that for every proper finite extension $L/K$ there exists $n\in \N$ such that $2^n\x \Psi_L$ is hyperbolic. 
Then $K$ is real closed and $\sign(\Psi)\neq 0$.
\end{lem}

\begin{proof}
By \Cref{P:lev} the field $K$ is real, by \Cref{C:pmqExt}  its only quadratic field extension is $K(\sqrt{-1})$, and
by \Cref{T:BFL} $K$ has no proper finite field extension of odd degree.  
Thus $K$ is real closed by \cite[Chap.~3, (2.3)]{Scharlau}.  It follows from \Cref{P:rcai} that $\Psi\simeq \Ad(\vf)\ox\Phi$ for a form $\vf$ over $K$ and a  
$K$-division algebra with canonical involution $\Phi$. 
As $K$ is real closed, 
it follows with \Cref{E:itr} that $\itr(\Phi)$ is positive definite, and thus $\sign(\Psi)$  is equal to  $|\sign(\vf)|$ or to $2\cdot |\sign(\vf)|$.
As $\Psi$ is not hyperbolic, $\vf$ is not hyperbolic, and we conclude that $\sign(\Psi)\neq 0$.
\end{proof}

\begin{lem}\label{L:nilpotent-awi}
Assume that $\Psi$ is split and let $r\in\nat$.
Then $\Psi^{\ox 2r}$ is hyperbolic if and only if $\itr(\Psi)^{\ox r}$ is hyperbolic.
\end{lem}
\begin{proof}
Replacing $\Psi$ by $\Psi^{\ox r}$ we may  in view of \Cref{P:itr-prod}  assume that $r=1$.
If $\Psi$ is symplectic then $\Psi$ and $\itr(\Psi)$ are hyperbolic by \Cref{C:trivial-hyper} and \Cref{C:hyper-trace}.
If $\Psi$ is orthogonal, then
$\Psi^{\ox 2}\simeq \Ad(\itr(\Psi))$ by \Cref{P:awi-square}, implying the statement.
Assume now that $\Psi$ is unitary.
Then $\Psi\simeq \Ad(\vf)\ox (a)_K$ for a form $\vf$ over $K$ and some $a\in\mg{K}$.
We obtain that $\Psi^{\ox 2}\simeq \Ad(\vf\ox\vf)\ox (a)_K$ and
$\itr(\Psi)\simeq \vf\ox\vf\ox \lla a\rra$.
Using \Cref{T:JAC} we conclude that $\Psi^{\ox 2}$ is hyperbolic if and only if $\itr(\Psi)$ is hyperbolic.
\end{proof}

\begin{thm}
\label{T:PLULG}
The following are equivalent:
\begin{enumerate}[$(i)$]
\item $\sign(\Psi)=0$;
\item $\Psi$ is weakly hyperbolic;
\item $2^n\times \Psi$ is hyperbolic for some $n\in\nat$;
\item either  $\Psi^{\otimes m}$ is hyperbolic for some $m\geq 1$, or $K$ is nonreal and $\Psi$ is split orthogonal of odd degree.
\end{enumerate}
These conditions are trivially satisfied if $K$ is nonreal.
\end{thm}

\begin{proof}
Trivially  $(iii)$ implies $(ii)$, and by \Cref{P:sign-mult}  any of the conditions implies $(i)$.

Suppose that $2^n\x \Psi$ is not hyperbolic for any $n\in\N$. 
By Zorn's Lemma there exists a maximal algebraic extension $L/K$ such that $2^n\x \Psi_L$ is not hyperbolic for any $n\in\nat$. By \Cref{L:zk} $L$ is real closed and $\Psi_L$ has nonzero signature at the unique ordering of $L$.
For the ordering $P=L^2\cap K$ of $K$ we  obtain that  $\sign_P(\Psi)\neq 0$. This shows that $(i)$ implies $(iii)$.

To finish the proof we show that $(i)$ implies $(iv)$.
We may assume that $\Psi$ is simple as otherwise $\Psi$ is hyperbolic. 
We may choose a positive integer $e$ such that $\Psi^{\ox e}$ is split.
Note that if $\Psi^{\ox e}$ is orthogonal of odd degree, then so is $\Psi$ and in particular $\Psi$ is split.
In view of \Cref{P:sign-mult}, we may now replace $\Psi$ by $\Psi^{\ox e}$ and assume that $\Psi$ is split.
From the assumption $(i)$ we obtain that $\sign(\itr(\Psi))=0$.
Note further that $\dim_K(\Psi)=\dim_K(\itr(\Psi))$.
If $\dim (\itr(\Psi))$ is odd, we conclude that $K$ is nonreal and $\Psi$ is split orthogonal of odd degree.
If $\dim (\itr(\Psi))$ is even, we obtain by \Cref{C:PLGP}
that $\itr(\Psi)^{\ox r}$ is hyperbolic for some positive integer $r$, and then $\Psi^{\ox 2r}$ is hyperbolic by \Cref{L:nilpotent-awi}.
\end{proof}

The equivalence $({ii}\Leftrightarrow{iii})$ in \Cref{T:PLULG} is equivalent to Scharlau's result \cite[Theorem~5.1]{Scha70} that the torsion of the Witt group of hermitian forms over an algebra with involution is $2$-primary.
The equivalence $(ii\Leftrightarrow iv)$ was observed in \cite[Proposition~5.41]{Unger}.

\begin{cor}
Assume that $\Z(\underline{\Psi})\simeq K(\sqrt{a})$ with $a\in\sums{K}$ in case $\Psi$ is unitary, and otherwise that
$\underline{\Psi_{R}}\sim  (-1,-\type(\Psi))_{R}$ for every real closure $R$ of $K$.
Then $\Psi$ is weakly hyperbolic.
\end{cor}
\begin{proof} 
In view of the hypothesis, we obtain from \Cref{P:rcai} that $\itr(\Psi)$ becomes hyperbolic over every real closure of $K$.
Therefore $\sign(\Psi)=0$, and it follows by \Cref{T:PLULG} that $\Psi$ is weakly hyperbolic.
\end{proof}

Let $T$ be a preordering of $K$.
We say that a $K$-algebra with involution $\Psi$ is \emph{$T$-hyperbolic} if there exists a $T$-positive quadratic form $\tau$ over $K$ such that $\Ad(\tau)\ox\Psi$ is hyperbolic. 
It is clear from \Cref{P:ad-mult} that a quadratic form $\vf$ over $K$ is $T$-hyperbolic if and only if $\Ad(\vf)$ is $T$-hyperbolic.

\begin{thm}\label{T:PLULG-Pre}
Let $T$ be a preordering of $K$. 
We have $\sign_P(\Psi)=0$ for every $P\in X_T(K)$ if and only if $\Psi$ is $T$-hyperbolic. Moreover, in this case there exists a $T$-positive Pfister form $\vt$ over $K$ such that $\Ad(\vt)\ox \Psi$ is hyperbolic.
\end{thm}

\begin{proof}
Assume first that $\Psi$ is $T$-hyperbolic.
Let $\vt$ be a $T$-positive form  over $K$ such that $\Ad(\vt)\ox \Psi$ is hyperbolic.
Using \Cref{P:sign-mult} it follows that $\sign_P(\vt)\cdot \sign_P(\Psi)=0$ for any ordering $P$ of $K$. For $P\in X_T$ we have $\sign_P(\vt)>0$ as $\vt$ is $T$-positive, and we conclude that $\sign_P(\Psi)=0$.

Assume now that $\sign_P(\Psi)=0$ for every $P\in X_T$.
Then $\vt\otimes \itr(\Psi)$ is hyperbolic for some $T$-positive Pfister form $\vt$ over $K$, by \Cref{T:pre-Pfister}. 
By \Cref{P:itr-prod} and \Cref{P:trace-qf-square} we have  $\itr(\Ad(\vt)\otimes \Psi)\simeq \vt\ox\vt\ox \itr(\Psi)$.
We conclude that $\Ad(\vt)\otimes \Psi$ has trivial total signature. 
By \Cref{T:PLULG} there exists $n\in\nat$ such that $2^n\times \Ad(\vt)\ox \Psi$ is hyperbolic.
Hence, the isomorphic $K$-algebra with involution $\Ad(2^n\times \vt)\ox \Psi$ is hyperbolic.
As $2^n\times \vt$ is a $T$-positive Pfister form, this shows the statement.
\end{proof}

\section{Bounds on the torsion order} 
\label{sec9}

Let $\Psi$ be a $K$-algebra with involution.
By \Cref{T:PLULG}, if $\Psi^{\ox n}$ is hyperbolic for some $n\in\nat$, we have that $2^m\times \Psi$ is  hyperbolic for some $m\in\nat$.
In this situation, one may want to bound $m$ in terms of $n$ and the degree of $\Psi$.
We restrict to the case $n=2$, that is, where $\Psi^{\ox 2}$ is hyperbolic, and use the function $\Delta:\nat\lra\nat$ introduced in Section 2 to bound $m$.

\begin{thm}
\label{P:gkar}
Assume that $\Psi^{\ox 2}$ is split hyperbolic.
Let $m=\deg(\Psi)$ if $\s$ is orthogonal or unitary, and $m=\frac{1}{2}\deg(\Psi)$ if $\s$ is symplectic.
Then $2^{\Delta(m)}\x \Psi$ is hyperbolic.
\end{thm}

\begin{proof}
In view of \Cref{T:KT} it suffices to consider the situation where $\Psi$ is either split orthogonal, or split unitary, or symplectic of index $2$.
Then by \Cref{P:JAC} we have that 
$\Psi\simeq \Ad(\vf)\ox \Phi$ for a  form $\vf$ over $K$ with $\dim(\vf)=m$ and a $K$-algebra with canonical involution $\Phi$.
As $\Psi^{\ox 2}$ is hyperbolic, it follows from \Cref{P:awi-square} if $\Psi$ is orthogonal or symplectic and from \Cref{L:nilpotent-awi} if $\Psi$ is unitary
that $\itr(\Psi)$ is hyperbolic.
By \Cref{P:itr-prod} and \Cref{P:trace-qf-square}
we have $\itr(\Psi)\simeq \vf\ox\vf\ox\itr(\Phi)$.
Hence, \Cref{P:Lou} yields that the form $(2^{\Delta(m)}\times  \vf)\ox\itr(\Phi)$ is hyperbolic.
 We conclude using \Cref{T:JAC} that $2^{\Delta(m)}\times \Psi\simeq \Ad(2^{\Delta(m)}\times \vf)\ox\Phi$ is hyperbolic.
\end{proof}

\begin{thm}
\label{T:quatord}
Assume that $\underline{\Psi}$ is a $K$-quaternion algebra and let $m\in\nat$ be such that $2^m \x \Psi^{\ox 2}$ is hyperbolic.
Then $2^{m+1}\x \Psi$ is hyperbolic.
Furthermore, if $\Psi$ is symplectic, then $2^m \x \Psi$ is hyperbolic.
\end{thm}

\begin{proof} 
Assume that $\Psi$ is split and orthogonal  or  unitary.
Then either
$\Psi\simeq \Ad\lla a\rra$ for some $a\in\mg{K}$ or $\Psi\simeq \Ad\lla a\rra\ox (b)_K$ for some $a,b\in \mg{K}$. Either way,
as $\lla a, a\rra \simeq 2\x \lla a \rra$
it follows that $\Psi^{\ox 2}\simeq 2\x \Psi$, whereby 
 $2^m \x \Psi^{\ox 2}\simeq 2^{m+1}\x \Psi$ and the statement follows.
We derive the implication claimed in general for $\Psi$ orthogonal or unitary by reduction to the split case by means of \Cref{T:KT}.

Assume now that $\Psi$ is symplectic. If $\Psi$ is split then it is hyperbolic and so is $2^m\times\Psi$.
Assume that $\Psi$ is not split. Then we have $\Psi^{\ox 2}\simeq \Ad(\itr(\Psi))$ by \Cref{P:awi-square} and hence $2^m\x \Psi^{\ox 2}\simeq \Ad(2^m\times \itr(\Psi))$ by \Cref{P:ad-mult}.
Hence, if $2^m\x \Psi^{\ox 2}$ is hyperbolic, then $2^m\times \itr(\Psi)$ is hyperbolic, and it follows by \Cref{T:JAC} that  $2^m\times \Psi$ is hyperbolic.
\end{proof}

The following example shows that the converse in \Cref{T:quatord} does not hold in general.

\begin{ex}
Let $m\in\nat$.
Assume that $K$ is either $k(t)$ or $k(\!(t)\!)$ for some field~$k$. Let $a \in \D_k (2^{m+1})\setminus \D_k (2^m)$. 
Then the form $2^m\times \lla a, -t\rra$ over $K$ is anisotropic.
Therefore $2^m\times (a \qil t)_K^{\ox 2} \simeq \Ad(2^m\x \lla a, -t\rra)$ is anisotropic, whereas $2^{m+1}\x (a \qil t)_K$ is hyperbolic.
\end{ex}

\begin{thm}\label{P:final}
Let $a,b\in \mg K$. The $K$-algebra with involution $2\x (a\qil b)_K$ is hyperbolic if and only if $a\in \D_K \la 1,1\ra \cup \D_K\la 1,b\ra$.
For $n\in\nat $ with $n\geq 2$ we have that $2^n\x (a\qil b)_K$ is hyperbolic if and only if 
$a=x(y+b)$ with $x\in \D_K(2^n-1)$ and $y \in \D_K(2^n-1)\cup \{0\}$. 
\end{thm}

\begin{proof}
Note that
$2\x (a \qil b)_K \simeq (-1 \qil 1)_K \ox (a \qil b)_K \simeq (-1 \qilr a)_K\ox (a \qilr -b)_K$ by \Cref{P:flipflop}.
Hence, by \Cref{P:bqhyp}, $2\x (a \qil b)_K$ is hyperbolic if and only if one of $(-1,a)_K$ and $(a,-b)_K$ is split, which happens if and only if
$a\in \D_K\la 1,1\ra \cup \D_K\la 1,b\ra$.

Let  $n\geq 2$. Let $L$ denote the function field of the conic $aX^2+bY^2=1$ over~$K$.
Note that $(a,b)_L$ is split and thus $2^n\x (a\qil b)_L \simeq \Ad(2^n\times\lla a\rra_L)$.
Using \Cref{T:DPSS} we obtain that $2^n\x (a \qil b)_K$ is hyperbolic  if and only if  $2^n\x (a \qil b)_L$ is hyperbolic, which is 
the case if and only if $2^n\x\lla a\rra_L$ is hyperbolic.
Using \cite[Chap.~X, (4.28)]{Lam} we conclude that this happens if and only if $\la 1,-a,-b\ra$ is a subform of $2^n\x \lla a\rra$ over $K$. (Note that, since $n\geq 2$ both conditions hold in particular when $2^n\times\lla a\rra$ is hyperbolic; for $n=1$ this would fail.)
This is the case if and only if $(2^n-1)\times \lla a\rra\perp\la b\ra$ is isotropic.
Finally, this occurs if and only if $a=x(y+b)$ for some $x\in \D_K(2^n-1)$ and $y \in \D_K(2^n-1)\cup \{0\}$. 
\end{proof}

\begin{qu}
If $K$ is pythagorean and $\sign(\Psi)=0$, is $2\times \Psi$ then necessarily hyperbolic?
\end{qu}

\section*{Acknowledgements}
This work was supported by the \emph{Deutsche Forschungsgemeinschaft} (project \emph{Quadratic Forms and Invariants}, {BE 2614/3}), by the \emph{Zu\-kunfts\-kolleg, Universit\"at Konstanz},  by the Odysseus Programme 
funded by the \emph{Fonds Wetenschappelijk Onderzoek -- Vlaanderen}
(project \emph{Explicit Methods in Quadratic Form Theory}), and by the \emph{Science Foundation Ireland} Research Frontiers Programme (project no. 07/RFP/MATF191).


\section*{References}

\begin{thebibliography}{99}


\bibitem{BST}
E. Bayer-Fluckiger, D.B. Shapiro, J.-P. Tignol.
\newblock Hyperbolic involutions.
\newblock {\em Math. Z.} {\bf 214} (1993), 461--476.

\bibitem{BB} 
B. Bernauer.
\newblock {\em Ein Lokal-Global-Prinzip f\"ur Involutionen und hermitesche Formen}. 
\newblock Diplomarbeit, Universit\"at Konstanz, 2004.\\
\newblock {\em Online:} {\tt http://kops.ub.uni-konstanz.de/handle/urn:nbn:de:bsz:352-opus-13143}

\bibitem{Dej} I. Dejaiffe. Formes antihermitiennes devenant hyperboliques sur un corps de d{\'e}ploiement. \emph{C.~R.~Acad.~Sci., Paris, S{\'e}r. I, Math.} \textbf{332} (2001), 105--108.

\bibitem{Conc} R.L. Graham, D.E. Knuth, O. Patashnik. \emph{Concrete mathematics. A foundation for computer science.} Addison-Wesley Publishing Company, Advanced Book Program, Reading, MA, 1989.


\bibitem{Jac40}
N. Jacobson.
\newblock A note on hermitian forms.
\newblock {\em Bull. Amer. Math. Soc.} {\bf 46} (1940), 264--268.

\bibitem{Kar} N.A.~Karpenko. 
Hyperbolicity of orthogonal involutions. {\em Documenta Math., Extra Volume  Suslin} (2010), 371--389.


\bibitem{BOI} M.-A. Knus, A.S. Merkurjev, M. Rost, J.-P. Tignol,
\emph{The Book of Involutions}, Coll. Pub.~{\bf 44}, Amer. Math. Soc.,
Providence, RI, 1998.

\bibitem{Lam} T.Y. Lam,  \emph{Introduction to quadratic forms over fields}, Graduate Studies in Mathematics  {\bf 67},  American Mathematical Society,  Providence,  RI, 2005.

\bibitem{LamOVQ}
T.Y. Lam.
\newblock {\em Orderings, valuations and quadratic forms}.
\newblock CBMS Regional Conf. Ser. Math., 52,
\newblock published by AMS, 1983.

\bibitem{Lew} D.W. Lewis. Witt rings as integral rings.  \emph{Invent. Math.}  \textbf{90}  (1987),  631--633.

\bibitem{Lew89}
D.W.~Lewis. New proofs of the structure theorems for Witt rings. \emph{Expo. Math.} \textbf{7} (1989), 83--88.

\bibitem{LT} D.W. Lewis, J.-P. Tignol, On the signature of an
involution, \emph{Arch. Math.} \textbf{60} (1993), 128--135.

\bibitem{LU}
D.W. Lewis, T. Unger.
\newblock A local-global principle for algebras with involution and hermitian forms.
\newblock {\em Math. Z.} {\bf 244} (2003), 469--477.


\bibitem{PSS} R.~Parimala, R.~Sridharan, V.~Suresh. Hermitian analogue of a theorem of Springer. \emph{J. Algebra} \textbf{243} 
(2001), 780--789.

\bibitem{Pfi66}
A. Pfister.
\newblock Quadratische Formen in beliebigen K\"orpern.
\newblock {\em Invent. Math.} {\bf 1} (1966), 116--132.


\bibitem{Queg}
A. Qu\'eguiner.
\newblock Signature des involutions de deuxi\`eme esp\`ece.
\newblock {\em Arch. Math.} {\bf 65} (1995), 408--412.

\bibitem{Scharlau}
W. Scharlau.
\newblock {\em Quadratic and {H}ermitian forms}.
\newblock Grundlehren 270, Springer, Berlin, 1985.

\bibitem{Scha70}
W. Scharlau.
\newblock Induction theorems and the structure of the Witt group.
\newblock {\em Invent. Math.} {\bf 4} (1970), 37--44.



\bibitem{Tignol}
J.-P.~Tignol. Hyperbolicity of symplectic and unitary involutions. Appendix to a paper of N.
Karpenko. {\em Documenta Math., Extra Volume  Suslin} (2010),  389--392.

\bibitem{Unger}
T.~Unger, Quadratic Forms and Central Simple Algebras with Involution, Ph.D. Thesis, National University of Ireland, Dublin, 2000.


\end{thebibliography}
\end{document}